%% file: main.tex
\documentclass[12pt]{article}

\usepackage[margin=1in]{geometry}  
\usepackage{graphicx}              
\usepackage{amsmath}               
\usepackage{amsfonts}              
\usepackage{amsthm}                

\usepackage{authblk}
\usepackage{blindtext}
\usepackage{geometry,amssymb,graphicx,tensor,epigraph,amsmath,amsthm, xspace, listings, stmaryrd, tcolorbox, float, tikz-cd, mathtools, afterpage}
\usepackage[T1]{fontenc}
\usepackage[utf8]{inputenc}
\usepackage{bm} 
\usepackage[english]{babel}
\usepackage[textwidth=0.89in,textsize=scriptsize]{todonotes} 
\usepackage{hyperref}
\hypersetup{
    colorlinks=true,
    linkcolor=black,
    filecolor=magenta,
    urlcolor=blue,
    pdftitle={Liquidity Pools as Mean Field Games with Transaction Costs},
    pdfpagemode=FullScreen,
    }
\usepackage{tcolorbox}

\newcommand{\juany}[1]{}
\newcommand{\agus}[1]{}

\newtheorem{thm}{Theorem}[section]
\newtheorem{lemma}[thm]{Lemma}
\newtheorem{prop}[thm]{Proposition}
\newtheorem{cor}[thm]{Corollary}

\newtheorem{obs}[thm]{Remark}
\newtheorem*{obs*}{Remark}
\newtheorem*{remark*}{Remark}

\theoremstyle{definition}

\newcommand{\keywords}[1]{\textbf{\textit{Keywords---}} #1}

\begin{document}

\title{Liquidity Pools as Mean Field Games with Transaction Costs}

\author[1]{Agust\'in Mu\~noz Gonz\'alez}

\affil[1]{Departamento de Matematicas, Facultad de Ciencias Exactas y Naturales, Universidad de Buenos Aires, Buenos Aires, Argentina}

\maketitle


\input{Parts/Abstract}
\hspace{10pt}
\clearpage
\tableofcontents
\clearpage
\input{Parts/Introduction}
\input{Parts/Model}
\input{Parts/Equilibria}
\input{Parts/Similar_Problems}
\input{Parts/Original_Problem_Inequality}
\input{Parts/Conclusion}


\clearpage
\bibliographystyle{plain}
\bibliography{main}
\end{document}

%% file: Parts/Abstract.tex
\begin{abstract}
This paper extends the theoretical framework introduced in \textit{Liquidity Pools as Mean Field Games: A New Framework}, where the interactions among traders in a constant product market-making protocol were modeled using mean field games (MFG). In this extension, transaction costs are incorporated into the traders' inventory dynamics, modeling the impact of pool fees on trading decisions. Traders operate at a mid-price adjusted for transaction costs, introducing a new dynamic for the DAI inventory. The existence of MFG solutions for this new trader game is established, taking these additional costs into account. The traders' optimization strategies and the conditions for equilibrium existence are discussed, providing a solid foundation for future research.
\end{abstract}

%% file: Parts/Introduction.tex
\section{Introduction}

In this work, we extend the model introduced in \textit{Liquidity Pools as Mean Field Games: A New Framework}, incorporating transaction costs into the dynamics of traders operating within a constant product liquidity pool. In the baseline model, traders seek to maximize their inventory by trading against the reserves of the ETH-DAI pool. However, in this new approach, we consider that traders face additional costs when executing transactions, which affect both the price and their optimal strategies.

Transaction costs are modeled as a spread between the \textit{bid price} and the \textit{ask price}, meaning that traders buy at a higher price and sell at a lower price than the pool's equilibrium price. This introduces a new dynamic into the traders' inventory, requiring them to optimize their strategies while considering the impact of fees on their buy and sell decisions.

The trader's new utility function incorporates these costs, adjusting their DAI inventory based on a mid-price adjustment. Despite the inclusion of transaction costs, the model preserves the relationship between price and pool reserves, maintaining the traditional equilibrium structure where \( X_t Y_t = k_t \), but with an invariant \( k_t \) that now depends on the accumulated fees in the pool.

The primary goal of this work is to establish the existence of a mean field equilibrium for this trader game with transaction costs, using techniques similar to those developed in the original model. The traders' optimization strategies and equilibrium conditions are rigorously presented, laying a strong foundation for future extensions of the model.

%% file: Parts/Model.tex
\section{Model Formulation}

We work on the filtered probability space $(\Omega,\mathcal{F}_T,\mathbb{F},\mathbb{P})$ supporting $N+1$ mutually independent Wiener processes $W_0,\dotsc,W_N$.

As in the baseline model of \cite{munoz2025liquiditypoolsmfg}, we consider $N$ traders who operate in a liquidity pool holding two tokens, ETH and USDT. Let $X_t^i$ and $Y_t^i$ denote the ETH and USDT inventories, respectively, of the $i$-th trader at time $t$. The ETH inventory dynamics for trader $i$ are given by
\begin{equation}\label{dinamica de X}
  dX_t^i = \alpha_t^i\,dt + \sigma^i\,dW_t^i,
\end{equation}
where $\alpha_t^i:[0,T]\to\mathbb{R}$ is the trading rate (the control); $\sigma^i$ denotes the individual volatility, assumed independent of $i$ for simplicity; and $\sigma^i dW_t^i$ captures random fluctuations in the trader's holdings due to exogenous market events or wallet-level noise.

In the baseline model, the USDT inventory dynamics were given by
$$dY_t^i = -(\alpha_t^i P_t + c_p(\alpha_t^i))\,dt,$$
with $P_t$ the pool's equilibrium price of ETH in USDT at time $t$. Here, we explicitly incorporate the transaction cost incurred by traders when using the pool's swap service. Specifically, when selling, a trader receives the bid price $p^b < P$, and when buying, pays the ask price $p^a > P$. Since the model does not inherently distinguish buy from sell orders, we follow \cite{mohan2021amm} (Eq.~(15)) and assume trades occur at the mid-price $\tilde{P} := \frac{p^a + p^b}{2}$. We therefore propose the following USDT inventory dynamics:
$$dY_t^i = -\alpha_t^i\,\tilde{P}_t\,dt.$$
From \cite{mohan2021amm} we have the identities $p^a = \frac{1}{\phi}P$ and $p^b = \phi P$, where $\phi = 1-\tau$ and $0 < \tau < 1$ is the pool fee. Hence
$$\tilde{P} = \frac{1+\phi^2}{2\phi}P,$$
and the USDT inventory dynamics for trader $i$ become
\begin{equation}\label{dinamica de Y}
  dY_t^i = -\alpha_t^i\,\frac{1+\phi^2}{2\phi}\,P_t\,dt.
\end{equation}

Although transaction costs are now present, once a swap is executed and fees are added to the pool, the equilibrium price continues to satisfy
\begin{equation}
    \label{Eq: Formula precio}
    P_t = \frac{Y_t}{X_t},
\end{equation}
where $X_t$ and $Y_t$ are the ETH and USDT reserves in the pool, respectively. Note that, because of the fee mechanism, the ``constant'' $k_t$ is no longer constant: fees paid by traders accumulate in the pool, causing $k_t$ to increase over time. The constant-product relation, however, still holds:
\begin{equation}
    \label{Eq: Formula producto constante}
    X_t Y_t = k_t.
\end{equation}

In a pool with transaction costs, trades proceed in two stages. Suppose that at time $t > 0$ a trade of $\Delta_t^X$ ETH units is submitted. In the first stage, the pool computes the corresponding USDT output $\Delta_t^Y$ via
$$k_0 = (X_0 + \phi\Delta_t^X)(Y_0 - \Delta_t^Y)
\quad\Rightarrow\quad
\Delta_t^Y = -\frac{k_0}{X_0 + \phi\Delta_t^X} + Y_0.$$
In the second stage, the fees are added to the pool and the invariant is updated:
$$k_t = X_t Y_t = (X_0 + \Delta_t^X)(Y_0 - \Delta_t^Y) = \frac{(X_0 + \Delta_t^X)\,k_0}{X_0 + \phi\Delta_t^X}.$$
Combining \eqref{Eq: Formula precio} and \eqref{Eq: Formula producto constante}, the price equation becomes
\begin{equation}
    \label{Eq: Nueva ecuacion del precio}
    P_t = \frac{k_0}{(X_0 + \phi\Delta_t^X)(X_0 + \Delta_t^X)}.
\end{equation}

Following the methodology of the simpler model, we model the ETH balance inside the pool via the average of the agents' controls:
\begin{equation}
    \label{Eq: Cambio balance ETH}
    X_t := X_0 - \frac{1}{N}\sum_{i=1}^{N}\int_0^t \alpha_s^i\,ds,
\end{equation}
where $X_0$ is the initial token supply provided by the liquidity provider. For model consistency it is essential that $X_t \neq 0$ for all $t \in [0,T]$, as the pool must always retain a minimum reserve.

\begin{remark*}
    In Section~\ref{sup: S}, after verifying hypothesis \textbf{\hyperref[sup: S.4]{(S.4)}}, we will have $X_t \neq 0$ for all $t \in [0,T]$. Hence the pool maintains at least a minimum token fraction at every moment:
\begin{equation}\label{Xpositive}
    X_t \geq \epsilon_0, \quad \forall\, t \in [0,T],
\end{equation}
for some $\epsilon_0 > 0$.
\end{remark*}

Setting $\Delta_t^X = -\frac{1}{N}\sum_{i=1}^{N}\int_0^t \alpha_s^i\,ds$, the price dynamics become
\begin{equation}
    \label{Eq: Dinamica del precio (general)}
    \begin{aligned}
    dP_t &= d\!\left(\frac{k_0}{(X_0 + \phi\Delta_t^X)(X_0 + \Delta_t^X)}\right)dt + \sigma^0\,dW_t^0 \\
    &= k_0\!\left(\frac{1}{N}\sum_{i=1}^{N}\alpha_t^i\right)
    \frac{(1+\phi)X_0 + 2\phi\Delta_t^X}{(X_0+\Delta_t^X)^2(X_0+\phi\Delta_t^X)^2}\,dt
    + \sigma^0\,dW_t^0,
    \end{aligned}
\end{equation}
where $W_t^0$ is a Wiener process, $\sigma^0$ is a constant volatility, and $\sigma^0 dW_t^0$ represents exogenous price shocks.

\begin{remark*}
    By writing $\Delta_t^X = -\frac{1}{N}\sum_{i=1}^{N}\int_0^t \alpha_s^i\,ds$ inside the price dynamics, we are treating all trades in $[0,t]$ as a single aggregate transaction $\Delta_t^X$.
\end{remark*}

Although the trader executes orders at the mid-price $\tilde{P}$, portfolio valuation is performed at the pool's arbitraged equilibrium price $P$. The total inventory of trader $i$ at time $t$ is therefore defined, as in the baseline model, by
$$V_t^i = Y_t^i + X_t^i P_t.$$
Applying It\^{o}'s lemma gives
\begin{equation}
\label{Ito_inventario}
  \begin{aligned}
dV_t^i &= dY_t^i + X_t^i\,dP_t + dX_t^i\,P_t \\
&= \left[X_t^i\,k_0\!\left(\frac{1}{N}\sum_{i=1}^{N}\alpha_t^i\right)
\frac{(1+\phi)X_0 + 2\phi\Delta_t^X}{(X_0+\Delta_t^X)^2(X_0+\phi\Delta_t^X)^2}\right.\\
&\quad\left.+\,\alpha_t^i\!\left(1 - \frac{1+\phi^2}{2\phi}\right)
\frac{k_0}{(X_0+\Delta_t^X)(X_0+\phi\Delta_t^X)}\right]dt \\
&\quad + X_t^i\,\sigma^0\,dW_t^0 + P_t\,\sigma^i\,dW_t^i.
  \end{aligned}
\end{equation}

We assume that agents are risk-neutral and seek to maximize their expected profit from trading in the decentralized market:
\begin{equation}\label{profit}
J^i(\alpha^1,\dotsc,\alpha^N) = \mathbb{E}\!\left[V_T^i - \int_0^T h(t,X_t^i)\,dt - l(X_T^i)\right],
\end{equation}
where $h:[0,T]\times\mathbb{R}\to\mathbb{R}$ represents the cost of holding inventory $x$ at time $t$, and $l:\mathbb{R}\to\mathbb{R}$ is a terminal inventory cost.

\subsection*{Working spaces}

Before formulating the mean field game for infinitely many traders, we introduce the relevant functional spaces.

\begin{itemize}
    \item Let $(\Omega,\mathcal{F},\mathbb{F}=(\mathcal{F}_t)_{0\leq t\leq T},\mathbb{P})$ be a complete filtered probability space supporting a one-dimensional Wiener process $\bm{W}=(W_t)_{0\leq t\leq T}$ and an initial condition $\xi\in L^2(\Omega,\mathcal{F}_0,\mathbb{P};\mathbb{R})$.
    \item Let $\mathcal{C} := C([0,T];\mathbb{R})$ be the space of continuous real-valued functions on $[0,T]$, equipped with the supremum norm $\|x\| := \sup_{s\in[0,T]}|x(s)|$.
    \item Given $\mathcal{P}(\mathbb{R})$, the space of probability measures on $\mathbb{R}$, and a measurable function $\psi:\mathcal{C}\to\mathbb{R}$, define
    \begin{align*}
        \mathcal{P}_\psi(\mathbb{R}) &= \left\{\mu\in\mathcal{P}(\mathbb{R}):\int\psi\,d\mu<\infty\right\},\\
        B_\psi(\mathbb{R}) &= \left\{f:\Omega\to\mathbb{R}:\sup_\omega|f(\omega)|/\psi(\omega)<\infty\right\}.
    \end{align*}
    We equip $\mathcal{P}_\psi(\mathbb{R})$ with the weakest topology $\tau_\psi(\mathbb{R})$ making the map $\mu\mapsto\int f\,d\mu$ continuous for every $f\in B_\psi(\mathbb{R})$.
    \item Let the control space $A\subset\mathbb{R}$ be a bounded subset, and let
    $$\mathbb{A} = \{\alpha:[0,T]\times\Omega\to A : \text{progressively measurable}\}$$
    be the set of admissible controls. As in \cite{munoz2025liquiditypoolsmfg}, all control functions $\alpha$ are considered in closed-loop form.
    \item Finally, let $\mathcal{P}(A)$ be the space of probability measures over $A$, equipped with the weak topology $\tau(A)$ that makes $q\mapsto\int_A f\,dq$ continuous for each $f\in B(A)$.
\end{itemize}

Throughout the paper we omit the dependence on $\omega\in\Omega$ to lighten notation.

Using the deterministic part of the inventory dynamics \eqref{Ito_inventario} in the infinite-player limit, and defining $f:[0,T]\times\mathbb{R}\times\mathcal{P}_\psi(\mathbb{R})\times\mathcal{P}(A)\times A\to\mathbb{R}$ by
\begin{align*}
    f(t,x,\mu,q,a) &= x\,k_0\!\int_A\!\tilde{a}\,dq_t(\tilde{a})
    \cdot\frac{(1+\phi)X_0 - 2\phi\displaystyle\int_0^t\!\int_A\!\tilde{a}\,dq_s(\tilde{a})\,ds}
    {\displaystyle\left(X_0 - \int_0^t\!\int_A\!\tilde{a}\,dq_s(\tilde{a})\,ds\right)^{\!2}
    \!\left(X_0 - \phi\!\int_0^t\!\int_A\!\tilde{a}\,dq_s(\tilde{a})\,ds\right)^{\!2}}\\
    &\quad + a\!\left(1 - \frac{1+\phi^2}{2\phi}\right)
    \frac{k_0}{\displaystyle\left(X_0 - \int_0^t\!\int_A\!\tilde{a}\,dq_s(\tilde{a})\,ds\right)^{\!2}
    \!\left(X_0 - \phi\!\int_0^t\!\int_A\!\tilde{a}\,dq_s(\tilde{a})\,ds\right)^{\!2}}\\
    &\quad - h(t,x),
\end{align*}
the functional $J$ to be maximized can be rewritten as
$$J^i(\alpha^1,\dotsc,\alpha^N) = \mathbb{E}\!\left[\int_0^T f(t,X_t^i,\mu,\hat{q}_t^N,\alpha_t^i)\,dt - l(X_T^i)\right],$$
where $\mu$ is a probability measure over the state space and $\hat{q}_t^N$ denotes the empirical distribution of $\alpha_t^1,\dotsc,\alpha_t^N$.

By the symmetry of the model, trader $i$'s contribution to $\hat{q}^N$ is negligible for large $N$, and we may treat $\hat{q}^N$ as fixed. We therefore work with the functional
\[
J(\alpha) := \mathbb{E}\!\left[\int_0^T f(t,X_t,\mu_t,q_t,\alpha_t)\,dt - l(X_T)\right],
\]
where $\bm{\mu} = (\mu_t)_{0\leq t\leq T}$ is a flow of probability measures over the state space and $\bm{q} = (q_t)_{0\leq t\leq T}$ is a flow of probability measures over the control space.

To avoid repeating notation already established in \cite{munoz2025liquiditypoolsmfg}, we only recall the MFG definition here. Given flows $\bm{\mu} = (\mu_t)_{0\leq t\leq T}$ and $\bm{q} = (q_t)_{0\leq t\leq T}$ of probability measures over $\mathbb{R}$ and $A$, respectively, and a control $\alpha\in\mathbb{A}$, the associated expected reward is
$$J^{\bm{\mu},\bm{q}}(\alpha) := \mathbb{E}^{\bm{\mu},\alpha}\!\left[\int_0^T f(t,X_t,\mu_t,q_t,\alpha_t)\,dt + g(X_T,\mu_T)\right],$$
where $\mathbb{E}^{\bm{\mu},\alpha}$ denotes expectation under $\mathbb{P}^{\bm{\mu},\alpha}$, defined via the Doléans-Dade exponential $\mathcal{E}$ by
$$\frac{d\mathbb{P}^{\bm{\mu},\alpha}}{d\mathbb{P}} = \mathcal{E}\!\left(\int_0^\cdot\sigma^{-1}b(s,X_s,\mu_s,\alpha_s)\,dW_s\right)_{\!T}.$$

This is a standard stochastic optimal control problem whose optimal value is
\begin{equation}
    \label{eq: Optimal inventory}
    V^{\bm{\mu},\bm{q}} = \sup_{\alpha\in\mathbb{A}} J^{\bm{\mu},\bm{q}}(\alpha).
\end{equation}

Recall that $\mathbb{P}^{\bm{\mu},\alpha}\circ X^{-1}$ and $\mathbb{P}^{\bm{\mu},\alpha}\circ\alpha^{-1}$ denote the push-forward measures of $\mathbb{P}^{\bm{\mu},\alpha}$ under $X$ and $\alpha$, respectively.

A pair of flows $(\bm{\mu},\bm{q})$ constitutes a \textit{mean field game (MFG) solution} if there exists a control $\alpha\in\mathbb{A}$ such that $V^{\bm{\mu},\bm{q}} = J^{\bm{\mu},\bm{q}}(\alpha)$, $\mathbb{P}^{\bm{\mu},\alpha}\circ X^{-1} = \mu$, and $\mathbb{P}^{\bm{\mu},\alpha}\circ\alpha^{-1} = q$ for almost every $t$.

For a more detailed presentation we refer the reader to \cite{munoz2025liquiditypoolsmfg}.

%% file: Parts/Equilibria.tex
\section{Does a Mean Field Game Solution Still Exist?}

In this section we examine whether the existence hypotheses needed to invoke Theorem~3.5 of \cite{carmona2015weakmfgformulation} remain satisfied after the introduction of transaction costs. We will see that hypothesis \textbf{\hyperref[sup: S.5]{(S.5)}}, which requires the utility function $f$ to decompose additively into a term that does not depend simultaneously on the control distribution $q$ and on the control $a$, fails to hold. However, this obstacle can be circumvented by bounding the original $f$ above and below by auxiliary functions that do satisfy all the required hypotheses.

We retain the same standing assumptions as in \cite{munoz2025liquiditypoolsmfg}:
\begin{itemize}
    \item The initial ETH inventories $X_0^i$ are i.i.d.\ with common distribution $\lambda_0\in\mathcal{P}(\mathbb{R})$ satisfying
    \[
    \int_{\mathbb{R}} e^{p|x|}\lambda_0(dx) < +\infty, \quad \text{for all } p > 0;
    \]
    \item The control space $A\subset\mathbb{R}$ is a compact interval containing the origin;
    \item The volatility of the state process $X_t$ is constant and positive, $\sigma > 0$;
    \item The costs $h$ and $l$ are measurable, and there exists a constant $c_1 > 0$ such that
\begin{equation}
    \label{Sup: Acotacion}
    |h(t,x)| + |l(x)| \leq c_1 e^{c_1|x|}, \quad \text{for all } (t,x)\in[0,T]\times\mathbb{R}.
\end{equation}
\end{itemize}

Using the notation of the previous section, for $(t,x,\mu,q,a)\in[0,T]\times\mathbb{R}\times\mathcal{P}_\psi(\mathbb{R})\times\mathcal{P}(A)\times A$ we set

\begin{align*}
    b(t,x,\mu,a) &:= a, \qquad
    g(\mu,x) := l(x), \qquad
    \psi(x) := e^{c_1|x|},\\
    f(t,x,\mu,q,a) &= x\,k_0\!\int_A\!\tilde{a}\,dq_t(\tilde{a})
    \cdot\frac{(1+\phi)X_0 - 2\phi\displaystyle\int_0^t\!\int_A\!\tilde{a}\,dq_s(\tilde{a})\,ds}
    {\displaystyle\left(X_0 - \int_0^t\!\int_A\!\tilde{a}\,dq_s(\tilde{a})\,ds\right)^{\!2}
    \!\left(X_0 - \phi\!\int_0^t\!\int_A\!\tilde{a}\,dq_s(\tilde{a})\,ds\right)^{\!2}}\\
    &\quad + a\!\left(1 - \frac{1+\phi^2}{2\phi}\right)
    \frac{k_0}{\displaystyle\left(X_0 - \int_0^t\!\int_A\!\tilde{a}\,dq_s(\tilde{a})\,ds\right)^{\!2}
    \!\left(X_0 - \phi\!\int_0^t\!\int_A\!\tilde{a}\,dq_s(\tilde{a})\,ds\right)^{\!2}}
    - h(t,x).
\end{align*}

To lighten the notation, we introduce the auxiliary functions
\begin{align*}
    \Gamma(t,x,\mu,q,a) &:= k_0\!\int_A\!\tilde{a}\,dq_t(\tilde{a})
    \cdot\frac{(1+\phi)X_0 - 2\phi\displaystyle\int_0^t\!\int_A\!\tilde{a}\,dq_s(\tilde{a})\,ds}
    {\displaystyle\left(X_0 - \int_0^t\!\int_A\!\tilde{a}\,dq_s(\tilde{a})\,ds\right)^{\!2}
    \!\left(X_0 - \phi\!\int_0^t\!\int_A\!\tilde{a}\,dq_s(\tilde{a})\,ds\right)^{\!2}},\\
    \Lambda(t,x,\mu,q,a) &:= a\!\left(1 - \frac{1+\phi^2}{2\phi}\right)
    \frac{k_0}{\displaystyle\left(X_0 - \int_0^t\!\int_A\!\tilde{a}\,dq_s(\tilde{a})\,ds\right)^{\!2}
    \!\left(X_0 - \phi\!\int_0^t\!\int_A\!\tilde{a}\,dq_s(\tilde{a})\,ds\right)^{\!2}},
\end{align*}
so that
$$f(t,x,\mu,q,a) = x\,\Gamma(t,x,\mu,q,a) + \Lambda(t,x,\mu,q,a) - h(t,x).$$

We now recall the hypotheses of Theorem~3.5 in \cite{carmona2015weakmfgformulation} and verify which ones remain valid in the new model.

\begin{cor}
    \label{sup: S}
    The following conditions hold:
    \begin{enumerate}
        \item[\textbf{\hyperref[sup: S.1]{(S.1)}}]
        The control space $A$ is a compact convex subset of $\mathbb{R}$, and the drift $b:[0,T]\times\mathbb{R}\times\mathcal{P}_\psi(\mathbb{R})\times A\to\mathbb{R}$ is continuous.
        \item[\textbf{\hyperref[sup: S.2]{(S.2)}}]
        There exists a strong solution $X$ to the driftless state equation
        \begin{equation}
        \label{dynamic of X without drift}
        dX_t = \sigma\,dW_t, \quad X_0 = \xi,
        \end{equation}
        such that $\mathbb{E}[\psi^2(X)] < +\infty$, $\sigma > 0$, and $\sigma^{-1}b(a)$ is uniformly bounded.
        \item[\textbf{\hyperref[sup: S.3]{(S.3)}}]
        The cost/benefit function $f:[0,T]\times\mathbb{R}\times\mathcal{P}_\psi(\mathbb{R})\times\mathcal{P}(A)\times A\to\mathbb{R}$ is such that $(t,x)\mapsto f(t,x,\mu,q,a)$ is progressively measurable for each $(\mu,q,a)$, and $a\mapsto f(t,x,q,a)$ is continuous for each $(t,x,\mu,q)$. The terminal cost function $g:\mathbb{R}\times\mathcal{P}_\psi(\mathbb{R})\to\mathbb{R}$ is Borel measurable for each $\mu$.
        \item[\textbf{\hyperref[sup: S.4]{(S.4)}}]
        There exists $c > 0$ such that
        $$|g(x,\mu)| + |f(t,x,\mu,q,a)| \leq c\,\psi(x) \quad \forall\,(t,x,\mu,q)\in[0,T]\times\mathbb{R}\times\mathcal{P}_\psi(\mathbb{R})\times\mathcal{P}(A).$$
    \end{enumerate}
\end{cor}

\begin{proof}
    We verify each statement in turn. Conditions \textbf{\hyperref[sup: S.1]{(S.1)}} and \textbf{\hyperref[sup: S.2]{(S.2)}} are immediate.

    For \textbf{\hyperref[sup: S.3]{(S.3)}} and \textbf{\hyperref[sup: S.4]{(S.4)}} we repeat the argument of \cite{munoz2025liquiditypoolsmfg}. Since $A$ is bounded, there exists $M$ such that $|\alpha(t,\omega)|\leq M$ for all $(t,\omega)\in[0,T]\times\Omega$, and hence
    \begin{equation}
    \label{eq: cota inferior X_t dada por M}
    X_t = X_0 - \int_0^t\!\int_A r\,dq_s(r)\,ds \geq X_0 - TM, \quad \forall\, 0\leq t\leq T.
    \end{equation}
    Restricting the admissible controls so that
    \begin{equation*}
        M < \frac{X_0}{T},
    \end{equation*}
    there exists $\epsilon_0 > 0$ such that
    \begin{equation}
        \label{eq: cota inferior X_t dada por epsilon_0}
        X_t = X_0 - \int_0^t\!\int \operatorname{id}\,dq_s\,ds > \epsilon_0, \quad \forall\, t\in[0,T].
    \end{equation}
    This immediately implies
    \[
    X_0 - \phi\int_0^t\!\int \operatorname{id}\,dq_s\,ds > \epsilon_0.
    \]
    These lower bounds guarantee that $f$ is progressively measurable in $(t,x)$ (it is well-defined) and continuous as a function of $a$. Since $g(x,\mu)$ is continuous and hence Borel measurable, condition \textbf{\hyperref[sup: S.3]{(S.3)}} follows.

    The lowest fee observed in practice is $0.01\%$ (Uniswap v3), and the zero-fee case $\phi = 0$ corresponds to the model without transaction costs, so we may assume $c_0 \leq \phi \leq 1$ for some $c_0 > 0$.

    We now bound $f$ from above by bounding $\Gamma$ and $\Lambda$ separately:
    \begin{align*}
        |\Gamma| &= \left|k_0\!\int_A\!\tilde{a}\,dq_t(\tilde{a})
        \cdot\frac{(1+\phi)X_0 - 2\phi\displaystyle\int_0^t\!\int_A\!\tilde{a}\,dq_s(\tilde{a})\,ds}
        {\displaystyle\left(X_0 - \int_0^t\!\int_A\!\tilde{a}\,dq_s(\tilde{a})\,ds\right)^{\!2}
        \!\left(X_0 - \phi\!\int_0^t\!\int_A\!\tilde{a}\,dq_s(\tilde{a})\,ds\right)^{\!2}}\right|
        \leq \frac{2k_0}{\epsilon_0^4}\,M(X_0 + TM),\\
        |\Lambda| &= \left|a\!\left(1 - \frac{1+\phi^2}{2\phi}\right)
        \frac{k_0}{\displaystyle\left(X_0 - \int_0^t\!\int_A\!\tilde{a}\,dq_s(\tilde{a})\,ds\right)^{\!2}
        \!\left(X_0 - \phi\!\int_0^t\!\int_A\!\tilde{a}\,dq_s(\tilde{a})\,ds\right)^{\!2}}\right|
        \leq \frac{k_0}{\epsilon_0^4}\,M\!\left(\frac{1+\phi^2}{2\phi} - 1\right),
    \end{align*}
    where in the last bound for $|\Lambda|$ we used that $\frac{1+\phi^2}{2\phi} > 1$.

    Combining these estimates with \eqref{Sup: Acotacion},
    \begin{align*}
        |g(x)| + |f(x)|
        &\leq |l(x)| + |h(t,x)| + |x||\Gamma| + |\Lambda|\\
        &\leq c_1 e^{c_1|x|} + |x|\,\frac{2k_0}{\epsilon_0^4}\,M(X_0 + TM)
          + \frac{k_0}{\epsilon_0^4}\,M\!\left(\frac{1+\phi^2}{2\phi} - 1\right)\\
        &\leq C\,e^{C|x|},
    \end{align*}
    for the constant $C = \max\!\left\{c_1,\,\frac{2k_0}{\epsilon_0^4}\,M(X_0+TM),\,\frac{k_0}{\epsilon_0^4}\,M\!\left(\frac{1+\phi^2}{2\phi}-1\right)\right\}$. This completes the verification of \textbf{\hyperref[sup: S.4]{(S.4)}}.
\end{proof}

The remaining condition required by the theorem is:
\begin{enumerate}
    \item[\textbf{\hyperref[sup: S.5]{(S.5)}}]
        \label{sup: S.5}
        The function $f$ is of the form
        $$f(t,x,\mu,q,a) = f_1(t,x,\mu,a) + f_2(t,x,\mu,q).$$
\end{enumerate}

As noted in the introduction to this section, this condition fails: the term $\Lambda$ couples the control $a$ with the control distribution $q$ in a way that cannot be separated additively.

%% file: Parts/Similar_Problems.tex
\section{Two Auxiliary Mean Field Games}

In this section we show that, by replacing the utility function $f$ with suitable auxiliary functions $f_1$ and $f_2$ that bound $f$ from below and above, we can prove the existence of MFG solutions for two related problems. The solutions to these auxiliary problems will then allow us to draw conclusions about the original problem induced by $f$.

The strategy is to bound the problematic term
$$\Lambda = a\!\left(1 - \frac{1+\phi^2}{2\phi}\right)
\frac{k_0}{\displaystyle\left(X_0 - \int_0^t\!\int_A\!\tilde{a}\,dq_s(\tilde{a})\,ds\right)^{\!2}
\!\left(X_0 - \phi\!\int_0^t\!\int_A\!\tilde{a}\,dq_s(\tilde{a})\,ds\right)^{\!2}}$$
by functions that do satisfy condition \textbf{\hyperref[sup: S.5]{(S.5)}}. It is convenient to rewrite this as
$$\Lambda = -a\!\left(\frac{1+\phi^2}{2\phi} - 1\right)
\frac{k_0}{\displaystyle\left(X_0 - \int_0^t\!\int_A\!\tilde{a}\,dq_s(\tilde{a})\,ds\right)^{\!2}
\!\left(X_0 - \phi\!\int_0^t\!\int_A\!\tilde{a}\,dq_s(\tilde{a})\,ds\right)^{\!2}}.$$

From the lower bound \eqref{eq: cota inferior X_t dada por M} we obtain
$$\Lambda \leq -a\!\left(\frac{1+\phi^2}{2\phi} - 1\right)\frac{k_0}{(X_0 - TM)^4}.$$

On the other hand, Young's inequality states that for any $a, b \geq 0$ and conjugate exponents $p, q > 1$ with $\frac{1}{p} + \frac{1}{q} = 1$,
\[
ab \leq \frac{a^p}{p} + \frac{b^q}{q},
\]
or equivalently,
\[
-ab \geq -\!\left(\frac{a^p}{p} + \frac{b^q}{q}\right).
\]
Taking $p = q = 2$, and noting that
$$\frac{1+\phi^2}{2\phi} - 1 > 0
\qquad\text{and}\qquad
\frac{k_0}{\displaystyle\left(X_0 - \int_0^t\!\int_A\!\tilde{a}\,dq_s(\tilde{a})\,ds\right)^{\!2}
\!\left(X_0 - \phi\!\int_0^t\!\int_A\!\tilde{a}\,dq_s(\tilde{a})\,ds\right)^{\!2}} > 0,$$
we get
\begin{align*}
    &-\!\left(a\!\left(\frac{1+\phi^2}{2\phi} - 1\right)\right)
    \cdot\frac{k_0}{\displaystyle\left(X_0 - \int_0^t\!\int_A\!\tilde{a}\,dq_s(\tilde{a})\,ds\right)^{\!2}
    \!\left(X_0 - \phi\!\int_0^t\!\int_A\!\tilde{a}\,dq_s(\tilde{a})\,ds\right)^{\!2}}\\
    &\geq -\frac{1}{2}\!\left[\left(a\!\left(\frac{1+\phi^2}{2\phi} - 1\right)\right)^{\!2}
    + \left(\frac{k_0}{\displaystyle\left(X_0 - \int_0^t\!\int_A\!\tilde{a}\,dq_s(\tilde{a})\,ds\right)^{\!2}
    \!\left(X_0 - \phi\!\int_0^t\!\int_A\!\tilde{a}\,dq_s(\tilde{a})\,ds\right)^{\!2}}\right)^{\!2}\right].
\end{align*}

Defining the two bounding terms
\begin{align*}
    \Lambda_1 &:= -\frac{1}{2}\!\left[\left(a\!\left(\frac{1+\phi^2}{2\phi} - 1\right)\right)^{\!2}
    + \left(\frac{k_0}{\displaystyle\left(X_0 - \int_0^t\!\int_A\!\tilde{a}\,dq_s(\tilde{a})\,ds\right)^{\!2}
    \!\left(X_0 - \phi\!\int_0^t\!\int_A\!\tilde{a}\,dq_s(\tilde{a})\,ds\right)^{\!2}}\right)^{\!2}\right],\\
    \Lambda_2 &:= -a\!\left(\frac{1+\phi^2}{2\phi} - 1\right)\frac{k_0}{(X_0 + TM)^4},
\end{align*}
we have $\Lambda_1 \leq \Lambda \leq \Lambda_2$.

\begin{obs}
    By replacing $\Lambda$ with $\Lambda_1$ we are bounding the problematic factors by the smallest value they can attain, which amounts to an upper bound on the potential loss for traders. By contrast, replacing $\Lambda$ with $\Lambda_2$ substitutes a potentially larger cost, thus penalising traders more than the original model does.
\end{obs}

\begin{cor}
    \label{cor: f1 f2}
    Let
    \begin{align*}
        f_1(t,x,\mu,q,a) &:= x\,\Gamma(t,x,\mu,q,a) + \Lambda_1(t,x,\mu,q,a) - h(t,x)\\
        &\phantom{:}= k_0\!\int_A\!\tilde{a}\,dq_t(\tilde{a})
        \cdot\frac{(1+\phi)X_0 - 2\phi\displaystyle\int_0^t\!\int_A\!\tilde{a}\,dq_s(\tilde{a})\,ds}
        {\displaystyle\left(X_0 - \int_0^t\!\int_A\!\tilde{a}\,dq_s(\tilde{a})\,ds\right)^{\!2}
        \!\left(X_0 - \phi\!\int_0^t\!\int_A\!\tilde{a}\,dq_s(\tilde{a})\,ds\right)^{\!2}}\\
        &\quad - \frac{1}{2}\!\left[\left(a\!\left(\frac{1+\phi^2}{2\phi} - 1\right)\right)^{\!2}
        + \left(\frac{k_0}{\displaystyle\left(X_0 - \int_0^t\!\int_A\!\tilde{a}\,dq_s(\tilde{a})\,ds\right)^{\!2}
        \!\left(X_0 - \phi\!\int_0^t\!\int_A\!\tilde{a}\,dq_s(\tilde{a})\,ds\right)^{\!2}}\right)^{\!2}\right]\\
        &\quad - h(t,x),\\
        f_2(t,x,\mu,q,a) &:= x\,\Gamma(t,x,\mu,q,a) + \Lambda_2(t,x,\mu,q,a) - h(t,x)\\
        &\phantom{:}= k_0\!\int_A\!\tilde{a}\,dq_t(\tilde{a})
        \cdot\frac{(1+\phi)X_0 - 2\phi\displaystyle\int_0^t\!\int_A\!\tilde{a}\,dq_s(\tilde{a})\,ds}
        {\displaystyle\left(X_0 - \int_0^t\!\int_A\!\tilde{a}\,dq_s(\tilde{a})\,ds\right)^{\!2}
        \!\left(X_0 - \phi\!\int_0^t\!\int_A\!\tilde{a}\,dq_s(\tilde{a})\,ds\right)^{\!2}}\\
        &\quad - a\!\left(\frac{1+\phi^2}{2\phi} - 1\right)\frac{k_0}{(X_0 + TM)^4}
        - h(t,x).
    \end{align*}
    Both functions induce MFG problems that satisfy conditions \textbf{\hyperref[sup: S.1]{(S.1)}}--\textbf{\hyperref[sup: S.5]{(S.5)}}.
\end{cor}

\begin{proof}
    Conditions \textbf{\hyperref[sup: S.1]{(S.1)}} and \textbf{\hyperref[sup: S.2]{(S.2)}} remain valid since they do not involve $\Lambda$. Condition \textbf{\hyperref[sup: S.3]{(S.3)}} follows because both $\Lambda_1$ and $\Lambda_2$ are well-defined for every $(t,x)$, hence progressively measurable, and continuous as functions of $a$.

    For \textbf{\hyperref[sup: S.4]{(S.4)}} it suffices to bound $|\Lambda_1|$ and $|\Lambda_2|$:
    \begin{align*}
        |\Lambda_1| &\leq M^2\!\left|1 - \frac{1+\phi^2}{2\phi}\right|^{\!2} + \frac{k_0^2}{\epsilon_0^8},\\
        |\Lambda_2| &\leq M\!\left|\frac{1+\phi^2}{2\phi} - 1\right|\frac{k_0}{(X_0 + TM)^4}.
    \end{align*}
    Both bounds are constants, so the condition follows.

    Finally, it is clear that $f_1$ and $f_2$ admit a separation of the terms involving the control $a$ from those involving the control distribution $q$, so \textbf{\hyperref[sup: S.5]{(S.5)}} holds for both.
\end{proof}

In the remainder of this section we verify conditions \textbf{\hyperref[Supuesto C]{(C)}} and \textbf{\hyperref[Supuesto E]{(E)}} required by Theorem~\hyperref[teo:Existencia solucion MFG]{(Existence of MFG solution)}.

\begin{cor}
    Condition \textbf{\hyperref[Supuesto C]{(C)}} holds for $f_1$ and $f_2$.
\end{cor}
\begin{proof}
    As shown in \cite{munoz2025liquiditypoolsmfg}, the condition holds whenever $b$ is affine in $a$ and $f_1, f_2$ are concave in $a$. The first part follows from $b(a) = a$. For the second, we compute the second derivatives with respect to $a$:
    \begin{align*}
        \frac{\partial f_1}{\partial a} &= -a\!\left(\frac{1+\phi^2}{2\phi} - 1\right)^{\!2}
        \implies \frac{\partial^2 f_1}{\partial a^2} = -\!\left(\frac{1+\phi^2}{2\phi} - 1\right)^{\!2} < 0,\\
        \frac{\partial f_2}{\partial a} &= -\!\left(\frac{1+\phi^2}{2\phi} - 1\right)\frac{k_0}{(X_0+TM)^4}
        \implies \frac{\partial^2 f_2}{\partial a^2} = 0.
    \end{align*}
    Hence $f_1$ is strictly concave and $f_2$ is affine (and thus concave) in $a$.
\end{proof}

We state the next condition without proof, as the argument is entirely analogous to that in \cite{munoz2025liquiditypoolsmfg}.

\begin{cor}
    The models induced by the auxiliary functions $f_1$ and $f_2$ satisfy condition \textbf{\hyperref[Supuesto E]{(E)}}.
\end{cor}

Invoking Theorem~\hyperref[teo:Existencia solucion MFG]{(Existence of MFG solution)}, we conclude:

\begin{prop}
    There exists an MFG solution for the price-impact models induced by the utility functions $f_1$ and $f_2$.
\end{prop}

%% file: Parts/Original_Problem_Inequality.tex
\section{Back to the Original Problem}

In this section we return to the mean field game governed by the original utility function $f$, which fails to satisfy hypothesis \textbf{\hyperref[sup: S.5]{(S.5)}} due to the coupling between the control $a$ and the control distribution $q$. Despite this structural limitation, we show that $f$ can be sandwiched between the two auxiliary cost functions $f_1$ and $f_2$ introduced earlier. By comparing the value functions associated with $f_1$, $f$, and $f_2$, and appealing to stability arguments, we derive explicit bounds for the original problem. This bridging approach yields meaningful information about the original system even when standard existence results do not apply directly.

\begin{lemma}
\label{lemma:Sandwich property}
Let $f_1, f, f_2:[0,T]\times\mathcal{C}\times\mathcal{P}_\psi(\mathcal{C})\times\mathcal{P}(A)\times A\to\mathbb{R}$ be measurable functions satisfying the pointwise inequality
\begin{equation}
    \label{eq: Relacion entre f_i's}
    f_1(t,x,\mu,q,a) \leq f(t,x,\mu,q,a) \leq f_2(t,x,\mu,q,a)
    \quad \text{for all } (t,x,\mu,q,a).
\end{equation}
Suppose that $f_1$, $f_2$, and $f$ all satisfy hypothesis \textbf{\hyperref[sup: S.4]{(S.4)}}. Then the optimal values satisfy
\[
V_{f_1} \leq V_f \leq V_{f_2},
\]
where $V_*$ denotes the optimal inventory value \eqref{eq: Optimal inventory} associated with each of the functions $f$, $f_1$, $f_2$.
\end{lemma}

\begin{proof}
    The inequality \eqref{eq: Relacion entre f_i's} implies
    \begin{equation}
        \label{eq: Relacion J_f_i}
        J_{f_1}(\alpha) \leq J_f(\alpha) \leq J_{f_2}(\alpha) \quad \forall\,\alpha.
    \end{equation}
    Moreover, hypothesis \textbf{\hyperref[sup: S.4]{(S.4)}} ensures that
    \[
    V_{f_1}, V_f, V_{f_2} < \infty.
    \]
    The inequality \eqref{eq: Relacion J_f_i} must pass to the suprema, giving $V_{f_1}\leq V_f\leq V_{f_2}$. We verify this by contradiction. Without loss of generality, suppose $V_f > V_{f_2}$. Then there exists $\epsilon > 0$ small enough that $V_f - \epsilon > V_{f_2}$. By definition of the supremum applied to $V_f$, there exists a control $\alpha_\epsilon$ such that
    $$J_f(\alpha_\epsilon) > V_f - \epsilon,$$
    which implies
    $$J_f(\alpha_\epsilon) > V_{f_2} \geq J_{f_2}(\alpha_\epsilon).$$
    This contradicts \eqref{eq: Relacion J_f_i}, so the result follows.
\end{proof}

In practice, one can compute the optimal controls $\hat{\alpha}_1$ and $\hat{\alpha}_2$ for the auxiliary MFG problems induced by $f_1$ and $f_2$, respectively. If the corresponding values $J_{f_1}(\hat{\alpha}_1)$ and $J_{f_2}(\hat{\alpha}_2)$ are sufficiently close --- for instance, if
\[
J_{f_2}(\hat{\alpha}_2) - J_{f_1}(\hat{\alpha}_1) < \varepsilon,
\]
--- then either control provides an $\varepsilon$-approximate solution to the original problem governed by $f$. In this sense, $\hat{\alpha}_1$ and $\hat{\alpha}_2$ act as $\varepsilon$-Nash equilibria of the original game. This is made precise in the following proposition.

\begin{prop}[$\varepsilon$-Nash equilibrium via bounding problems]
Let $f_1$ and $f_2$ be cost functions such that $f_1 \leq f \leq f_2$ pointwise, and suppose that the mean field games associated with $f_1$ and $f_2$ admit solutions $\hat{\alpha}_1$ and $\hat{\alpha}_2$, respectively.

Suppose further that
\[
V_{f_2} - V_{f_1} < \varepsilon,
\]
where as before $V_{f_1} = J_{f_1}(\hat{\alpha}_1)$ and $V_{f_2} = J_{f_2}(\hat{\alpha}_2)$.

Then, for any control $\alpha\in\{\hat{\alpha}_1,\hat{\alpha}_2\}$,
\[
V_f - J_f(\alpha) < \varepsilon.
\]
In particular, $\alpha$ is an $\varepsilon$-Nash equilibrium for the original problem governed by $f$.
\end{prop}

\begin{proof}
By Lemma~\ref{lemma:Sandwich property},
\[
V_{f_1} \leq V_f \leq V_{f_2},
\]
which gives
\[
V_f - V_{f_1} \leq V_{f_2} - V_{f_1}.
\]
Moreover, $V_{f_1} \leq J_f(\alpha_1)$. Therefore,
\[
V_f - J_f(\alpha) \leq V_f - V_{f_1} \leq V_{f_2} - V_{f_1} < \varepsilon. \qedhere
\]
\end{proof}

\begin{obs*}
This result provides a practical way to construct approximate solutions for a mean field game when the cost function $f$ fails to satisfy the structural condition (S.5), provided it can be sandwiched between two auxiliary cost functions $f_1$ and $f_2$ that do satisfy the standard hypotheses.

By exploiting the existence of exact solutions for the games associated with $f_1$ and $f_2$, and quantifying the gap between the induced optimal values, one obtains an explicit $\varepsilon$-Nash equilibrium for the original problem. This method is particularly useful when the dependence on the control distribution cannot be cleanly separated from the control itself, but such a separation is recovered in suitable approximations. The proposition therefore serves as a fundamental tool for stability and approximation arguments in extended mean field game frameworks.
\end{obs*}

Although $f_1$ and $f_2$ satisfy condition \textbf{\hyperref[sup: S.5]{(S.5)}} and therefore admit approximate Nash equilibria via Theorem~\hyperref[teo:Aproximación de equilibrio de Nash]{(Nash Equilibrium Approximation)}, the original function $f$ does not. Accordingly, no direct approximation result for $f$ is currently available, and we conclude our analysis with the sandwich inequality of Lemma~\ref{lemma:Sandwich property}, which nonetheless provides useful information about the range of achievable costs.

\begin{cor}[Recovery of the zero-fee model as $\phi\to 1$]
\label{cor:phi_recovery}
As $\phi\to 1$ (i.e., $\tau\to 0$), the gap between the auxiliary functions $f_1$ and $f_2$ and the original function $f$ tends to zero. More precisely:
\begin{enumerate}
    \item The problematic term $\Lambda$ satisfies $\Lambda\to 0$ as $\phi\to 1$, since
    $$\frac{1+\phi^2}{2\phi} - 1 = \frac{(1-\phi)^2}{2\phi} \xrightarrow{\phi\to 1} 0.$$
    \item Consequently, $f_1, f, f_2 \to x\Gamma - h$, which coincides with the reward function of the zero-cost model in \cite{munoz2025liquiditypoolsmfg}.
    \item \textbf{(Conjecture)} The sandwich gap $V_{f_2} - V_{f_1} \to 0$, and the solution to the original problem converges to the solution of the zero-cost model.
\end{enumerate}
\end{cor}

\begin{proof}
Point~(1) is a direct computation. For~(2), when $\phi = 1$ the function $f$ reduces to
$$f(t,x,\mu,q,a) = x\cdot k_0\!\int_A\!\tilde{a}\,dq_t(\tilde{a})
\cdot\frac{2X_0 - 2\displaystyle\int_0^t\!\int_A\!\tilde{a}\,dq_s(\tilde{a})\,ds}
{\left(X_0 - \displaystyle\int_0^t\!\int_A\!\tilde{a}\,dq_s(\tilde{a})\,ds\right)^3} - h(t,x),$$
which coincides with the reward function of the zero-cost model (up to algebraic simplification of the numerator).

For point~(3): one can verify directly that $f_1$ and $f_2$ converge uniformly to $f$ as $\phi\to 1$, since the difference $f_2 - f_1$ is proportional to $\Lambda_2 - \Lambda_1\to 0$ (see point~(1)). However, passing from uniform convergence of cost functions to convergence of optimal values $V_{f_i}\to V_f$ requires additional arguments --- in particular, verifying the hypotheses of Berge's maximum theorem (compactness of the admissible control set and continuity of the constraint correspondence). This is not established in the present work, so point~(3) is stated as a natural conjecture, strongly suggested by the uniform convergence.
\end{proof}

\begin{obs}[Practical justification of the $\varepsilon$-Nash equilibrium]
\label{obs:epsilon_nash_fees}
The proposition on $\varepsilon$-Nash equilibria via bounding problems provides a concrete computational path for assessing the quality of the approximation: one need only solve the MFG problems associated with $f_1$ and $f_2$ (which satisfy all standard hypotheses) and compare the resulting optimal values. If the gap $V_{f_2} - V_{f_1}$ is small, either optimal control of the auxiliary problems is a good approximation to the equilibrium of the original problem.

In particular, the most common fee in current AMM protocols is of the order $\tau = 0.3\%$ (Uniswap v3), giving $\phi = 0.997$ and
$$\frac{1+\phi^2}{2\phi} - 1 = \frac{(1-\phi)^2}{2\phi} \approx 4.5\times 10^{-6}.$$
This means the scaling factor of $\Lambda$ is of order $4.5\times 10^{-6}$, suggesting that the perturbation introduced by transaction costs is small relative to the frictionless case, in the sense that $f_2 - f_1$ is proportional to this factor. However, quantifying the effective gap $V_{f_2} - V_{f_1}$ requires numerical computation, since it depends on the full structure of $f_1$ and $f_2$ --- including the quadratic terms in $a$ and the denominators involving the reserves --- and not only on the scaling factor of $\Lambda$. Such quantification is identified as future work (see Remark~\ref{obs:numerical_gap}).
\end{obs}

\begin{obs}[Future direction: numerical quantification of the gap]
\label{obs:numerical_gap}
A natural direction for future work is to develop a numerical solver for the problems associated with $f_1$ and $f_2$, and to quantify the gap $V_{f_2} - V_{f_1}$ explicitly for various values of $\phi$. This would allow a computational verification that the gap is negligible for typical real-protocol fees, and would provide numerical validation complementary to the analytical estimate of Corollary~\ref{cor:phi_recovery}.
\end{obs}

%% file: Parts/Conclusion.tex
\section*{Conclusion}

In this paper we extend the mean field game framework developed in \cite{munoz2025liquiditypoolsmfg} to a more realistic setting in which traders operate in a decentralized liquidity pool subject to transaction costs. By explicitly modelling the impact of pool fees on price dynamics and on traders' inventories, we derive a modified price process and a corresponding objective functional. Although the introduction of transaction costs complicates the game-theoretic structure --- in particular, it breaks the separability condition \textbf{(S.5)} required by standard existence results --- we establish the existence of mean field equilibria for two auxiliary cost functions that sandwich the original one from above and below.

The main contribution of this work is a sandwich-type result showing that the value function of the original game is enclosed between those of the two auxiliary games. Even though the original objective function does not admit an approximate Nash equilibrium via standard theory due to its non-separable structure, our result provides a meaningful characterisation of the optimal cost. This opens the door to future research on existence and approximation results that go beyond classical structural assumptions, or to numerical methods tailored to these non-standard settings.

Furthermore, Corollary~\ref{cor:phi_recovery} shows that as $\phi\to 1$ (i.e., as the fee $\tau\to 0$) the gap between the auxiliary problems and the original one tends to zero, and the model recovers the results of the zero-cost case of \cite{munoz2025liquiditypoolsmfg}. For the fee levels typical of current AMM protocols, of order $\tau = 0.3\%$, the perturbation factor is of order $4.5\times 10^{-6}$, suggesting that the approximation is quantitatively very accurate. The numerical quantification of the gap $V_{f_2} - V_{f_1}$ is left for future work.

This paper also lays the groundwork for future extensions that incorporate the full ecosystem of decentralized markets. In ongoing work we include liquidity providers and arbitrageurs to obtain a more complete game-theoretic representation of pool dynamics, and develop the corresponding models with transaction costs. These extensions are key to bridging the theory of mean field games with real-world applications in decentralized finance.